\documentclass{amsart}

\usepackage{amsmath,
             amssymb, amsfonts}

\newtheorem{theorem}{Theorem}[section]

\newtheorem{lemma}[theorem]{Lemma}

\theoremstyle{definition}
\newtheorem{definition}[theorem]{Definition}

\theoremstyle{remark}

\numberwithin{equation}{section}

\begin{document}

\title[Prime numbers for the parallelogram equation]{
Prime numbers as a uniqueness set  \\
of the parallelogram equation \\
via the Goldbach's conjecture}

\author{Hee Chul Pak*}
\address{Department of Mathematics,
         Dankook University, 119 Dandae-ro, Dongnam-gu, Cheonan, Chungnam, 31116,  Republic of Korea}
\email{hpak@dankook.ac.kr}
\thanks{*Correspondence: Hee Chul Pak, hpak@dankook.ac.kr}

\author{Dongseung Kang}
\address{Department of Mathematics Education,
         Dankook University, 152 Jukjeon, Suji, Yongin, Gyeonggi, 16890, Republic of Korea}
\email{dskang@dankook.ac.kr}

\begin{abstract}
Multiplicative arithmetic functions $f$
satisfying the  parallelogram functional equation on prime numbers:
\begin{equation*}
f(p+q)+f(p-q)=2f(p)+2f(q)
\qquad
 \mbox{for all primes\,\, } p,\,q  \; (p \geq q)
\end{equation*}
are investigated.
It is derived that
the unique solution
is the quadratic function
$f(n)=n^2$
for all $n \in \mathbb{N}$
by the Goldbach's conjecture.
\end{abstract}

\maketitle

{\small

\indent \hspace{.43in} Goldbach's conjecture
}

\section{Introduction}\label{Intro}

We consider the following parallelogram  functional equation: for $m \geq n$
\begin{equation}\label{eqn1}   
f(m+n)+f(m-n) = 2f(m) + 2f(n)
\end{equation}  
for arithmetic functions $f: \mathbb{N} \to \mathbb{C}$.
In  (\ref{eqn1}),
  we set that for $m = n$
\begin{equation}  \label{cond-1}
f(m-n) = f(0) = 0.
\end{equation}
The constraint
(\ref{cond-1}) simply
arises from the restricted domain of
arithmetic functions $f: \mathbb{N} \to \mathbb{C}$.
Instead of the requirement (\ref{cond-1}),
we may
replace the arithmetic solutions $f: \mathbb{N} \to \mathbb{C}$ for the
parallelogram equation (\ref{eqn1}) equipped  with (\ref{cond-1})
to
the solutions
$f: \mathbb{Z} \to \mathbb{C}$
 defined on the extended set $\mathbb{Z}$ of the whole integers for
the  equation (\ref{eqn1}) without  (\ref{cond-1}).
From a solution $f: \mathbb{Z} \to \mathbb{C}$,
 we can get an
arithmetic solution
$f|_{\mathbb{N}}: \mathbb{N} \to \mathbb{C}$ by restriction
for the  equation (\ref{eqn1}) with (\ref{cond-1}), and vice versa.
In fact, each solution $f: \mathbb{Z} \to \mathbb{C}$
for the equation  (\ref{eqn1})
possesses the even symmetry:
 $$
 f(-n) = f(n)
 $$
together with the fact $f(0) = 0$.
Therefore
the problem corresponding to the solutions $f: \mathbb{N} \to \mathbb{C}$
and
the problem of those $f: \mathbb{Z} \to \mathbb{C}$
are equivalent.

We are interested in finding and classifying the arithmetic functions $f$ satisfying
the function equation $(\ref{eqn1})$ on prime numbers $m$ and $n$.
In order to clarify our result, we first introduce a terminology inspired by \cite{Spiro}.
\begin{definition}   
Let $\mathcal{S}$ be a (partial) collection of arithmetic functions and $E$ be a subset of
the set $\mathbb{N}$ of all natural numbers.
Suppose there is exactly one element $f: \mathbb{N} \to \mathbb{C}$ in $\mathcal{S}$
which satisfies
the the parallelogram function equation (\ref{eqn1}) for all $m,\,n\in E$($m \geq n$). Then
$E$ is said to be  a {\it parallelogram uniqueness set for $S\,.$}
\end{definition}    

For example,  the set $\mathbb{N}$ of all natural numbers is a  parallelogram uniqueness set for
the set $\mathcal{S}_{M}$ of all multiplicative\footnote{The term
{\it multiplicative function} means
$f(1) = 1$ and $f(m n) = f(m)f(n)$ for all coprime integers $m$ and $n$.}
  arithmetic functions.
However, the set $4 \mathbb{N} := \{ 4 n  : n \in \mathbb{N}   \}$
 is not a parallelogram uniqueness set for
 $\mathcal{S}_M$.
Our  main goal is to demonstrate that
the set of all {\it prime numbers} is  a parallelogram uniqueness set for
the collection $\mathcal{S}_M$.

Prime numbers have aroused human curiosity since the beginning of human intellectual history.
One of the most interesting things about prime numbers is
their random distribution on the  line filled with the natural numbers
- even though  it appears to be a pattern on a large scale, they are not fully understood yet.
This is one of the main reasons why  the mathematical analysis on prime numbers are so hard.

Many conjectures have been raised about prime numbers.
Despite of seemingly elementary formulations, many of conjectures on prime numbers
have been unsolved for decades.
 One of the  most  famous conjectures among them is the Goldbach's conjecture,
 which states that every even integer
greater than $2$ can be written as the sum of two primes.

 With the advance of numerical analysis, the Goldbach's conjecture
  has been verified up to the  number $4 \times 10^{18}$
 (and double-checked up to $4 \times 10^{17}$) in the year of 2013.\footnote{
  Tom\'{a}s Oliveira e Silva, Goldbach conjecture verification,
 www.sweet.ua.pt/tos/goldbach.html}
It has been shown by several number theorists that {\it most} even numbers
are expressible as the sum of two primes. More precisely,
 the set of even integers
 that are not the sum of two primes has density zero.

In this  paper we prove that the set of all prime numbers is a parallelogram uniqueness set for
the collection $\mathcal{S}_M$ via the Goldbach's conjecture.
We state our main theorem as follows:
\begin{theorem}\label{main1}   
Let $f$ be a multiplicative arithmetic function.
 If
$f$ satisfies the  parallelogram equation on prime numbers, that is,
\begin{equation}\label{eqn2}
f(p+q)+f(p-q)=2f(p)+2f(q) \mbox{\,\,\,\,\, for all primes\,\, } p \geq q\,,
\end{equation}
then $f$ is the quadratic function  $f(n)=n^2$ for all $n \in \mathbb{N}$.
\end{theorem}          

We briefly point out the main algebraic structural ingredient for this  problem.
Theorem \ref{main1}  insists  that
for a multiplicative arithmetic function $f$,
 a discriminantal subset $\mathcal{A}_f$ of $\mathbb{N}$
defined by
\begin{equation}    \label{setA}
\mathcal{A}_f := \{ n \in \mathbb{N} : f(n) = n^2   \}
\end{equation}
is the whole set $\mathbb{N}$.
The main issue on the induction process
is to argue whether $p \in \mathcal{A}_f$ and  $q \in \mathcal{A}_f$
imply $p+q \in \mathcal{A}_f$ or not.
In simpler cases, for example, the problem of the additive equation:
\[
f(p) + f(q) = f(p+q)
\]
has the property  $p+q \in \bar{\mathcal{A}}_f := \{ n \in \mathbb{N} : f(n) = n  \}$
if $p$ and $q$ are primes(see \cite{Spiro}).\footnote{In fact,
for the additive equation, primes are members of $\bar{\mathcal{A}}_f$.}
However this may not be the case for the problem of
 parallelogram equation   (\ref{eqn2}).
In fact, the main difficulty is to demonstrate
the fact that
 the subtraction $p-q$ (with $p \geq q$) is
in the member of  $\mathcal{A}_f$.
To get over this obstruction,
we  downsize the magnitude of the subtraction $p-q$ (see
the end of the proof of the main theorem  together with
Lemma \ref{lem33} and Lemma \ref{lem44} for the structure of the problem).

Throughout this paper, all numbers are positive integers.
For example, $m$, $n$, $p$, $q$, $r$, $p_1$, $q_1$, $a$, $b$, $N$, $\cdots$
always stand  for positive integers.

\section{Arguments}   \label{main}

We  present Theorem~\ref{main1} by using the Goldbach's conjecture.
In the following,   $f$ is a multiplicative arithmetic function
equipped with (\ref{eqn2}).

Since $f$ is a multiplicative function,
 we have $f(1)=1$, and also $f(0) = 0$ by
the constraint (\ref{cond-1}).
  It can be easily checked that
\begin{align*}
f(4)&= 4f(2)  \\
f(5)&= 2f(3)+2f(2)-1\\
f(6)&=  f(2)f(3)   \\
f(7)&= 3f(3)+6f(2)-2  \\
&\cdots,
\end{align*}
which require the  following lemma:
\begin{lemma}
We have $f(2)=4$ and $f(3) = 9$.
\end{lemma}
\begin{proof}
Suppose that $f(2)\neq 4$.
Since $f$ is a multiplicative function and
satisfies the parallelogram equation (\ref{eqn2}), we have
$$
f(2)f(p)=f(2p)=f(p+p)=4f(p)
$$
for any odd prime $p$.
Then $f(p)=0$ for any odd prime $p\,$, which presents
that $f(3) = f(5)=0$. From the  equation that
$f(5)=2f(3)+2f(2)-1$,  we should have $f(2)=\frac{1}{2}$.
Also, $f(3) = f(7)=0$ and  the  fact that $0=f(7)=3f(3)+6f(2)-2$ imply
that $f(2)=\frac{1}{3}\,.$
 This is a contradiction. Thus we conclude that $f(2)=4$.

We also  have $f(3)=9$
by the the facts that  $f(4) =  4f(2) = 16$ and
$$
f(3)f(4)=f(12)=15f(2)+10f(3)-6=10f(3)+54
$$
by the multiplicativity  of $f$.
\end{proof}
The exact values of $f$ for small numbers can be obtained:
\begin{eqnarray*}
&&f(2)=2^2,\,f(3)=3^2,\, f(4)=4^2,\,f(5)=2f(3)+2f(2)-1=5^2, \\
&&f(6)=f(2)f(3)=6^2,\,f(7)=3f(3)+6f(2)-2=7^2,\, f(8)=8^2,\\
&&f(9)=4f(3)+12f(2)-3=9^2,\,f(10)=f(2)f(5)=10^2.
\end{eqnarray*}
In fact, we confirm that
\begin{lemma}\label{lem20}
For the natural number $n \leq 20$, one has $f(n)=n^2$.
\end{lemma}
\begin{proof}
For example,  since $f$ is a multiplicative function,
we have
$$
f(2)f(7)=f(14)=2f(11)-46
$$
by plugging $p=11$ and $q=3$ in the equation~(\ref{eqn2})
to get
$$
f(11)=11^2 \mbox{\,\, and \,\,} f(14)=14^2.
$$
We also note that
\[
f(4) f(5) = f(20) = 2 f(17) + 2f(3) - f(14)
\]
and $f(14) = f(2)f(7)$ to get the exact  value of  $f(17)$.
\end{proof}

With the notation in (\ref{setA}), we have:
\begin{lemma}\label{lem33}
Let $p$ and $q$ be primes with $p \geq q\,.$
If three of four numbers $p+q\,,$ $p-q\,,$ $p$ and $q$ are elements in the set $\mathcal{A}_f$,
then the rest is contained in  $\mathcal{A}_f$.
\end{lemma}
For example, if $p, p+q$ and $p-q$ belong to
$\mathcal{A}_f$, then $q \in \mathcal{A}_f$.
\begin{proof}
Suppose that
$p, q$ and $p+q$ are in $\mathcal{A}_f$. Then
plugging the values
$f(p) = p^2$, $f(q) = q^2$ and $f(p+q) = (p+q)^2$
into the parallelogram equation
\[
f(p+q)+f(p-q)=2f(p)+2f(q),
\]
we have $f(p-q) = 2p^2 + 2q^2 -(p+q)^2 = (p-q)^2$.
Hence  $p-q \in \mathcal{A}_f$.
All  the other cases can be  dealt with similarly.
\end{proof}

We also observe:
\begin{lemma}\label{lem44}
Let $a$ and $b$ be relatively prime.
If two of $a$, $b$ and $a \times b$
are elements in the set $\mathcal{A}_f$,
then the other one is in  $\mathcal{A}_f$.
In particular, we have that
\begin{align} \label{fact-0}
a, b \in \mathcal{A}_f \; \; {\rm implies } \; \;
a \times b \in \mathcal{A}_f.
\end{align}
\end{lemma}
\begin{proof}
If $a$, $a \times b$ are in $\mathcal{A}_f$, then one has
\[
a^2 f(b) = f(a)f(b) = f(a \times b) = (a\times b)^2
\]
to get $f(b) = b^2$. If $a, b \in \mathcal{A}_f$, then
$f(a \times b) = f(a)f(b) = a^2 \times b^2$.
\end{proof}

\medskip

Based on Lemma \ref{lem20},  Lemma \ref{lem33} and  Lemma \ref{lem44},
we now present the proof of the main theorem.

\noindent{\it Proof of Theorem \ref{main1}. }
It suffices to show that $\mathcal{A}_f = \mathbb{N}$.
The mathematical induction is utilized to show that
$ \mathbb{N} \subset \mathcal{A}_f$.
Clearly, $1 \in  \mathcal{A}_f$.
We assume that $n \in \mathcal{A}_f$  for any $n$ with $n \leq N$.
We will show that $N+1 \in  \mathcal{A}_f$.
By Lemma~\ref{lem20}, it is enough to consider the case $N \geq 20$.   \\

We consider two cases separately: The first case is
that $N+1$ is a power of prime, that is,
\begin{align} \label{c-1}
N +1 = \bar{p}^e
\end{align}
for some prime $\bar{p}$ and natural number $e$. The other case is
that $N+1$ is not a power of  single prime, and hence $N+1$ can be written as
\begin{align}  \label{c-2}
 N+1 = a \times b
\end{align}
 for some coprime $a, b > 1$; that is,  $(a, b) = 1$.  \\

We will first clear up the easier case (\ref{c-2}):
$N+1 = a\times b$.
The facts $a \leq N$ and $b \leq N$
imply that $f(a) = a^2$ and $f(b) = b^2$, and so
by the multiplicativity  of $f$, we get
\[
f(N+1) = f(a)f(b)= (ab)^2 = (N+1)^2.
\]
Therefore we see that $N+1 \in \mathcal{A}_f$.  \\

We now consider the case (\ref{c-1}): $N+1 = \bar{p}^e$. We divide this case into two subcases: indeed,
either $\bar{p} = 2$ or $\bar{p}$ is an odd prime.

If $N+1 = 2^e$, then the Goldbach's conjecture yields that
there are two prime numbers $p$ and $q$ such that
$$
N+1=p+q
$$
with $p \geq q$.
Then
since $p$, $q$ and $p-q$ are  all less than equal to $N$,
we have that  $p, q, p-q \in \mathcal{A}_f$.
Hence
Lemma~\ref{lem33} leads to $N+1=p+q \in  \mathcal{A}_f$.  \\

Now we consider the case that
 $$
 N+1 = \bar{p}^e
 $$
where $\bar{p}$ is an odd prime.
We divide this case into the two subcases: either $e = 1$ or $e \neq 1$.
This  division is mainly to secure the strict inequality at (\ref{ineq-1}) below.
If $e = 1$, then  $N+1 = \bar{p}$ is a prime.
In this case, we select $q$ either $q = 3$ or  $q = 5$ satisfying
$$
(N+1)+q \equiv 2\, ({\rm mod } \; 4).
$$
Then we note that
by the choice of $q$, $2$ and  $\frac{N+1+q}{2}$ are relatively prime and so
we have that
\begin{equation}\label{eqn21}
f\left( N+1+q \right)
=
f(2)f\left(\frac{N+1 +q}{2}\right)
\end{equation}
by the multiplicativity  of $f$.
Since  $\frac{N+1 +q}{2} \leq N$ (because  $N \geq 6$),
we have
$$
\frac{N+1 +q}{2} \in \mathcal{A}_f.
$$
 Thus
$N+1 +q \in \mathcal{A}_f$ by Lemma \ref{lem44}.
Therefore we have that
\[
 N+1 +q, \;
 N+1 - q
 \; \;  {\rm and }  \; \;  q
\]
(because $N+1 - q \leq N$) are elements  in $\mathcal{A}_f$.
Hence
Lemma \ref{lem33} yields that
$$
N+1 \in \mathcal{A}_f.
$$

\medskip

It remains to consider the case $N +1  = \bar{p}^e$ with $e > 1$.
By  the Goldbach's conjecture,
there are prime numbers $p$ and $q$ such that
$$
2(N+1)=p+q
$$
and $p \geq q$.  We will demonstrate
\begin{equation}\label{check-point}
p, q, p-q \in \mathcal{A}_f
\end{equation}
to get $2(N+1) \in \mathcal{A}_f$
by Lemma \ref{lem33}. This directly implies the desired result:
$$
N+1 \in \mathcal{A}_f
$$
because
$N+1$ is  an odd number, and
\[
4 f(N+1) = f(2)f(N+1) = f(2(N+1))
 = (2(N+1))^2.
\]
This completes the proof.
Hence the remaining proof is devoted to verify the conditions (\ref{check-point}).  \\

First, we observe that $N+1$ is the arithmetic mean of $p$ and $q$ to present
\begin{align}  \label{ineq-00}
 q\leq N+1 \leq p.
\end{align}
From the fact that $N+1$ is not a prime number, (\ref{ineq-00}) says
\begin{align}  \label{ineq-1}
q < N+1,
\end{align}
which in turn implies $q \in \mathcal{A}_f$.  \\

We next show that $p \in \mathcal{A}_f$.
Since $p$ is an odd number,
we can choose a prime number  $r$
among $\{3,\,5,\,7,\,17\}$ to get
\begin{equation}\label{eqn22}
p+r \equiv 4\,({\rm mod } \; 8).
\end{equation}
 By the choice  of $r$ obeying  (\ref{eqn22}),
 the number  $4$ and the {\it integer} $\frac{p+r}{4}$
are relatively prime.
We also note that
$$
4N = 2(p+q-2) \geq p+r
$$
for $p \geq N+1 > 20$ to see  that the integer $\frac{p+r}{4}$ is in $\mathcal{A}_f$.
Hence
$p +r \in \mathcal{A}_f$ by Lemma \ref{lem44}.
The equation~(\ref{eqn22})  also implies that
$\frac{p-r}{2}$ is an odd integer, and
$$
\frac{p-r}{2}< N.
$$
Hence
$p - r \in \mathcal{A}_f$ by Lemma \ref{lem44}.
We have  shown that $p+r, p-r$ and $r$ are in $\mathcal{A}_f$.
Therefore $p \in \mathcal{A}_f$ by Lemma \ref{lem33}.  \\

It remains to check whether $p-q$ is in  $\mathcal{A}_f$  or not.
Since $p-q$ is even, we factorize $p-q$ as
$$
p-q = 2^e\,a\,\,\mbox{ where }\, e\geq 1\, \mbox{ and }\, 2 \nmid a\,.
$$
If $a \neq 1\,,$ then in view of $2^e\,a=p-q \leq 2N-1$,
 both $2^e$ and $a$
are less than or  equal to  $N$.
Hence,
by Lemma \ref{lem44}, we get $p-q \in \mathcal{A}_f$.

The remaining case is when $p-q = 2^e$.
By  the Goldbach's conjecture,
there are prime numbers $p_1$ and $q_1$ such that
$$
2^e=p_1+q_1
$$
 with  $p_1 \geq q_1$.
We will verify
\begin{equation}\label{cond-2}
p_1, q_1, p_1 - q_1   \mbox{ are  members of }   \mathcal{A}_f.
\end{equation}
From the fact that
$q_1 \leq \frac{p_1 + q_1}{2}  =  \frac{ p-q}{2}  \leq  N - \frac{1}{2}$,
we have $q_1 \in   \mathcal{A}_f$.
We next show that  $p_1 \in \mathcal{A}_f$.
 Since $p_1$ is an odd number, we can choose a prime number  $r$
among $\{3,\,5,\,7,\,17\}$ to get
\begin{equation}\label{eqn222}
p_1+r \equiv 4\,({\rm mod } \; 8).
\end{equation}
The same argument used above leads that
$$
\frac{p_1+r}{4} \leq N \,\,
\mbox{ and }\,\,
\left( 4,
\frac{p_1+r}{4}
\right) = 1
$$
to have $p_1 + r \in \mathcal{A}_f$ by Lemma \ref{lem44}.
Similarly, the  {\it integer} $\frac{p_1- r}{2}$ satisfies
$
\frac{p_1- r}{2} \leq N
$
 and
$
\left( 2,
\frac{p_1-r}{2}
\right) = 1
$
to get $p_1 - r \in \mathcal{A}_f$ by Lemma \ref{lem44}.
Since
 $p_1 +r$, $p_1 -r$ and $r$ are in $\mathcal{A}_f$, we find
 $p_1 \in \mathcal{A}_f$ by Lemma \ref{lem33}.

Finally, we claim that $p_1-q_1$ is in  $\mathcal{A}_f$.
Since $p_1-q_1$ is even, we factorize $p_1-q_1$ as
$$
 p_1-q_1=2^{e_1}\,a_1
$$
for some $e_1 \geq 1$ and $a_1$ with $2\nmid a_1$.
Similar to the previous case, we have that
if $a_1 \neq 1$, then $p_1-q_1 \in \mathcal{A}_f$.
 If $a_1=1\,,$ then
$$
 2^{e_1} = p_1-q_1 <   p_1 + q_1
 = 2^e \leq 2N-1,
$$
which implies $e_1 < e$. This says that  $p_1-q_1$ is less than
or  equal to
the half of  $p_1 + q_1$.
Therefore
 $p_1-q_1 \in \mathcal{A}_f$.
The conditions (\ref{cond-2})  are now verified.
Hence
 by Lemma~\ref{lem33},
$p-q=p_1+q_1 \in \mathcal{A}_f$ as desired.

All problems are now settled down.
\hfill$\Box$\par 

\bigskip

\bibliographystyle{amsalpha}

\end{document}